\documentclass[11pt]{elsarticle}

 \usepackage{graphics}

\usepackage{amssymb}

\usepackage{amsfonts}
\usepackage{pifont}
\usepackage{amsmath}
\usepackage{latexsym}
\usepackage{bm}
\usepackage{amssymb}
\usepackage{amsthm}
\usepackage{graphics}
\usepackage{epstopdf}
\usepackage{tikz}
\def\bi{\begin{align}}
\def\bin{\begin{align*}}
\newcommand{\ein}{\end{align*}}
\newcommand{\ei}{\end{align}}
\newcommand{\dif}{\mathrm{d}}

\newcommand{\sumli}{\sum\limits}
\newcommand{\ba}{\begin{array}}
\newcommand{\ea}{\end{array}}
\renewcommand{\vec}[1]{\bm{#1}}

\numberwithin{equation}{section} \numberwithin{figure}{section}
\numberwithin{table}{section}

\newtheorem{exa}{Example}[section]
\theoremstyle{definition}
\newtheorem{defi}{Definition}[section]
\theoremstyle{plain}
\newtheorem{lem}{Lemma}[section]
\theoremstyle{plain}
\newtheorem{thm}{Theorem}[section]
\newtheorem{coro}{Corollary}[section]

\journal{Elsevier}

\begin{document}

\begin{frontmatter}



\title{Construction of symplectic (partitioned) Runge-Kutta
methods with continuous stage}

\author[a]{Wensheng Tang\corref{cor1}}
\ead{tangws@lsec.cc.ac.cn}\cortext[cor1]{Corresponding author.}
\address[a]{College of Mathematics and Computational Science\\
Changsha University of Science and Technology\\ Changsha 410114, P.
R. China}
\author[a]{Guangming Lang}
\ead{langguangming1984@126.com}
\author[a]{Xuqiong Luo}
\ead{luoxuqiong@163.com}

\begin{abstract}
Hamiltonian systems are one of the most important class of dynamical
systems with a geometric structure called symplecticity and the
numerical algorithms which can preserve such geometric structure are
of interest. In this article we study the construction of symplectic
(partitioned) Runge-Kutta methods with continuous stage, which
provides a new and simple way to construct symplectic (partitioned)
Runge-Kutta methods in classical sense. This line of construction of
symplectic methods relies heavily on the expansion of orthogonal
polynomials and the simplifying assumptions for (partitioned)
Runge-Kutta type methods.
\end{abstract}

\begin{keyword}
Hamiltonian system; Symplectic method; Continuous-stage Runge-Kutta
method; Continuous-stage partitioned Runge-Kutta method.

\end{keyword}

\end{frontmatter}


\section{Introduction}
\label{}



Geometric numerical integration is a subfield of the numerical
solution of differential equations, and it turn out to be very
efficient for long-time simulating the dynamical behavior of those
systems with special structures \cite{hairerlw06gni}. For
Hamiltonian systems, it is known that symplecticity of the flow is
an important characteristic property of such systems
\cite{Arnold89mmo} and to correctly reproduce this property, the
numerical algorithms called symplectic methods have been proposed
and extensively studied (see
\cite{Feng84ods,sanzc94nhp,hairerlw06gni} and references therein),
among of which symplectic (partitioned) Runge-Kutta methods are one
of the most important subclass of such algorithms
\cite{sanz88rkm,suris89ctg,sun93coh,sun95coh}. In this article, we
focus on a new and simple way to construct these methods.

Runge-Kutta (RK) method with continuous stage has been discussed and
investigated by several authors recently
\cite{hairer10epv,Tangs12tfe,Tangs12ana,Tangs14cor,miyatake14aee},
the idea of which was firstly presented by Butcher in 1970s
\cite{butcher72aat,butcher87tna}. There are few literatures
involving this kind of methods since Butcher's seminal work, until
recently, \cite{hairer10epv} exploits it to interpret the
energy-preserving collocation methods. Subsequently,
\cite{Tangs12tfe} investigates four kinds of time finite element
methods and relates them to Runge-Kutta methods with continuous
stage. Moreover, it is shown in \cite{Tangs12tfe} that some
energy-preserving RK methods including $s$-stage trapezoidal method
\cite{Iavernarop07sst}, Hamiltonian boundary value methods
\cite{brugnanoit10hbv,brugnanoit15aoh}, average vector field method
\cite{quispelm08anc,Celledoni09mmoqw} can also be related to such
methods. \cite{Tangs12ana,Tangs14cor} discuss the construction of
Runge-Kutta methods with continuous stage and present the
characterizations for several geometric-structure preserving methods
including symplectic methods, symmetric methods and
energy-preserving methods. Exponentially-fitted continuous-stage
Runge-Kutta methods with energy-preserving property for Hamiltonian
systems are also presented in \cite{miyatake14aee}. In this paper,
we aim at further investigation of construction of symplectic
(partitioned) Runge-Kutta methods, along the line of expansion of
orthogonal polynomials and with the help of simplifying assumptions
for (partitioned) Runge-Kutta type methods. This approach can
naturally provides a simple way to construct symplectic
(partitioned) Runge-Kutta methods in classical sense.

The outline of the paper is as follows. In Section 2, we provide
some preliminaries for (partitioned) Runge-Kutta methods with
continuous stage and review some existing results for the
construction of general RK type methods with continuous stage. In
section 3, we discuss the construction of symplectic (partitioned)
Runge-Kutta methods with continuous stage as well as the symplectic
(partitioned) Runge-Kutta methods in classical sense. At last, we
conclude this paper.
\section{Construction of (P)RK type methods}

In this section, we first introduce the definition of (partitioned)
Runge-Kutta method with continuous stage by following the
formulation proposed in \cite{hairer10epv}, then show some results
which are useful in constructing (P)RK type methods.

\subsection{Continuous-stage RK method}
Consider a first-order system of ordinary differential equations
(ODEs)
\begin{equation}\label{eq:ode}
\begin{cases}
\vec{\dot{z}}=\vec{f}(t,\vec{z}),\\
\vec{z}(t_0)=\vec{z}_0\in \mathbb{R}^d,
\end{cases}
\end{equation}
where the upper dot indicates differentiation with respect to $t$
and
$\vec{f}:\mathbb{R}\times\mathbb{R}^{d}\rightarrow\mathbb{R}^{d}$ is
a sufficiently smooth vector function.

\begin{defi}\label{defi1}\cite{hairer10epv,Tangs14cor}
Let $A_{\tau,\, \sigma}$ be a function of two variables $\tau$,
$\sigma$ $\in [0, 1]$, $B_\tau$ be a function of $\tau\in [0, 1]$.
Define $C_\tau=\int_0^1A_{\tau,\,\sigma}\,\dif\sigma$. The one-step
method $\Phi_h: \vec{z}_0 \mapsto \vec{z}_{1}$ given by
\begin{equation}\label{crk}
\begin{split}
&\vec{Z}_\tau=\vec{z}_0+h\int_0^{1}A_{\tau,\,\sigma}\vec{f}(t_0+C_\sigma
h,\vec{Z}_\sigma)\,\dif
\sigma,\;\tau\in[0,\,1],\\
&\vec{z}_{1}=\vec{z}_0+h\int_0^{1}B_\tau\vec{f}(t_0+C_\tau
h,\vec{Z}_\tau)\,\dif\tau,
\end{split}
\end{equation}
is called a continuous-stage Runge-Kutta (csRK) method, where
$\vec{Z}_\tau\approx\vec{z}(t_0+C_\tau h),\,\vec{z}_{1}\approx
\vec{z}(t_0+h).$
\end{defi}
In this paper, the construction of csRK methods mainly relies on the
following simplifying assumptions of order conditions
\cite{hairer10epv}
\begin{equation*}
\begin{split}
&\mathfrak{B}(\xi):\quad \int_0^1B_\tau C_\tau^{\kappa-1}\,\dif
\tau=\frac{1}{\kappa},\quad \kappa=1,\ldots,\xi,\\
&\mathfrak{C}(\eta):\quad
\int_0^1A_{\tau,\,\sigma}C_\sigma^{\kappa-1}\,\dif
\sigma=\frac{1}{\kappa}C_\tau^{\kappa},\quad \kappa=1,\ldots,\eta,\\
&\mathfrak{D}(\zeta):\quad \int_0^1B_\tau C_\tau^{\kappa-1}
A_{\tau,\,\sigma}\,\dif
\tau=\frac{1}{\kappa}B_\sigma(1-C_\sigma^{\kappa}),\quad
\kappa=1,\ldots,\zeta.
\end{split}
\end{equation*}
Analogously to the classical result \cite{butcher64ipa} given by
Butcher in 1964, we can determine the order of a csRK method with
the help of these simplifying assumptions.
\begin{thm}\label{crk:order}
If a csRK method \eqref{crk} with coefficients
$(A_{\tau,\,\sigma},\,B_\tau,\,C_\tau)$ satisfies
$\mathfrak{B}(\rho)$, $\mathfrak{C}(\alpha)$ and
$\mathfrak{D}(\beta)$, then it is at least of order $$\min
(\rho,2\alpha+2,\,\alpha+\beta+1).$$
\end{thm}
\begin{proof}
This statement can be proved by using the similar idea to the one
used in Theorem 7.4  \cite{hairernw93sod}.
\end{proof}

According to Proposition 2.1 shown in \cite{Tangs14cor}, we assume
$B_\tau=1$ in the current paper and to proceed conveniently,
hereafter we also assume $C_\tau=\tau$. In such a case,
$\mathfrak{B}(\xi)$ is reduced to $$\int_0^1 \tau^{\kappa-1}\,\dif
\tau=\frac{1}{\kappa},\quad \kappa=1,\ldots,\xi$$ which implies
$\mathfrak{B}(\xi)$ holds with $\xi=\infty$. Now we introduce the
$\iota$-degree normalized shifted Legendre polynomial denoted by
$P_\iota(x)$, which can be explicitly computed by the Rodrigues
formula
$$P_0(x)=1,\;P_\iota(x)=\frac{\sqrt{2\iota+1}}{\iota!}\frac{{\dif}^\iota}{\dif x^\iota}
\Big(x^\iota(x-1)^\iota\Big),\; \;\iota=1,2,3,\cdots.$$ A well-known
property of Legendre polynomials is that they are orthogonal to each
other with respect to the $L^2$ inner product on $[0,\,1]$
$$\int_0^1 P_\iota(t) P_\kappa(t)\,\dif t= \delta_{\iota\kappa},\quad\iota,\,
\kappa=0,1,2,\cdots,$$
and they as well satisfy the following integration formulae
\begin{equation}\label{property}
\begin{split} &\int_0^xP_0(t)\,\dif
t=\xi_1P_1(x)+\frac{1}{2}P_0(x), \\
&\int_0^xP_\iota(t)\,\dif
t=\xi_{\iota+1}P_{\iota+1}(x)-\xi_{\iota}P_{\iota-1}(x),\quad
\iota=1,2,3,\cdots, \\
&\int_x^1P_\iota(t)\,\dif
t=\delta_{\iota0}-\int_0^{x}P_\iota(t)\,\dif t,\quad
\iota=0,1,2,\cdots,
\end{split}
\end{equation}
where $\xi_\iota=\frac{1}{2\sqrt{4\iota^2-1}}$ and $\delta_{ij}$ is
the Kronecker symbol.

With the help of \eqref{property}, by exploiting the expansion of
Legendre polynomials for the coefficients of a csRK method and the
corresponding simplifying assumptions, we have the following result
for construction of csRK methods (more details see
\cite{Tang13tfe,Tangs14cor}).

\begin{lem}\cite{Tang13tfe}
For a csRK method with $B_\tau\equiv1$ and $C_\tau=\tau$, we have
the following statements:
\begin{itemize}
\item[(a)] $\mathfrak{C}(\alpha)$ is equivalent to the fact that $A_{\tau,\,\sigma}$
takes the form
$$A_{\tau,\,\sigma}=\sum_{\iota=0}^{\alpha-1}\int_0^{\tau}P_\iota(x)\,\dif
x\,P_\iota(\sigma)+\sum_{\iota=\alpha}^\infty \gamma_\iota(\tau)
P_\iota(\sigma),$$ where $\gamma_\iota(\tau)$, $\iota\geq\alpha$ are
arbitrary functions;
\item[(b)] $\mathfrak{D}(\beta)$ is equivalent to the fact  that $A_{\tau,\,\sigma}$
takes the form
$$A_{\tau,\,\sigma}=\sum_{\iota=0}^{\beta-1}\int_{\sigma}^1P_\iota(x)\,\dif
x\, P_\iota(\tau)+\sum_{\iota=\beta}^\infty \lambda_\iota(\sigma)
P_\iota(\tau),$$ where $\lambda_\iota(\sigma)$, $\iota\geq\beta$ are
arbitrary functions.
\end{itemize}
\end{lem}

\begin{thm}\cite{Tangs14cor} \label{mainthm}
For a csRK method with $B_\tau\equiv1$ and $C_\tau=\tau$ (then
$\mathfrak{B}(\infty)$ holds), the following two statements are
equivalent:
\begin{itemize}
\item[(a)] Both $\mathfrak{C}(\alpha)$ and $\mathfrak{D}(\beta)$
hold;
\item[(b)] The coefficient $A_{\tau,\,\sigma}$ has the  form
\begin{equation}\label{coef}
A_{\tau,\,\sigma}=\frac{1}{2}+\sum_{\iota=0}^{N_1}\xi_{\iota+1}
P_{\iota+1}(\tau)P_\iota(\sigma)-\sum_{\iota=0}^{N_2}\xi_{\iota+1}
P_{\iota+1}(\sigma)P_\iota(\tau)+\sum_{i\geq\beta \atop
j\geq\alpha}\gamma_{ij}P_i(\tau)P_j(\sigma),
\end{equation}
where $N_1=\max(\alpha-1,\,\beta-2)$,
$N_2=\max(\alpha-2,\,\beta-1)$,
$\xi_\iota=\frac{1}{2\sqrt{4\iota^2-1}}$ and $\gamma_{ij}$ are
arbitrary real numbers.
\end{itemize}
\end{thm}
By combining Theorem \ref{crk:order} and Theorem \ref{mainthm} we
can easily construct lots of csRK methods, the order of which are
given by $\min (\infty,\,2\alpha+2,\,\alpha+\beta+1)=\min
(2\alpha+2,\,\alpha+\beta+1)$.

\subsection{Continuous-stage PRK method}
For a partitioned ODEs given by
\begin{equation}
\begin{cases}
\vec{\dot{y}}=\vec{f}(t,\vec{y},\vec{z}),\;\vec{y}(t_0)=\vec{y}_0\in
{\mathbb{R}}^{d},\\
\vec{\dot{z}}=\vec{g}(t,\vec{y},\vec{z}),\;\vec{z}(t_0)=\vec{z}_0\in
{\mathbb{R}}^{s},
\end{cases}
\end{equation}
we can consider methods that deal with the $\vec{y}$ and $\vec{z}$
variables by two different csRK methods, which then produce the
following definition.
\begin{defi}\cite{Tang13tfe,Tangs14cor}\label{defi2}
Assume that $C_\tau=\int_0^1A_{\tau,\,\sigma}\,\dif\sigma$ and
$\widehat{C}_\tau=\int_0^1\widehat{A}_{\tau,\,\sigma}\,\dif\sigma$.
The following one-step method
\begin{equation}\label{cprk}
\begin{split}
&\vec{Y}_\tau=\vec{y}_0+h\int_0^{1}A_{\tau,\,\sigma}\vec{f}(t_0+C_\sigma
h,\vec{Y}_\sigma,\vec{Z}_\sigma)\,\dif\sigma,\;\;\tau\in [0, 1],\\
&\vec{Z}_\tau=\vec{z}_0+h\int_0^{1}\widehat{A}_{\tau,\,\sigma}\vec{g}(t_0+C_\sigma
h,\vec{Y}_\sigma,\vec{Z}_\sigma)\,\dif\sigma,\;\;\tau\in [0, 1],\\
&\vec{y}_{1}=\vec{y}_0+h\int_0^{1}B_\tau\vec{f}(t_0+C_\tau
h,\vec{Y}_\tau,\vec{Z}_\tau)\,\dif\tau,\\
&\vec{z}_{1}=\vec{z}_0+h\int_0^{1}\widehat{B}_\tau\vec{g}(t_0+C_\tau
h,\vec{Y}_\tau,\vec{Z}_\tau)\,\dif\tau.
\end{split}
\end{equation}
is referred to as continuous-stage partitioned Runge-Kutta (csPRK)
method.
\end{defi}

Similarly as those order conditions considered for classical PRK
method \cite{cohenh11lei}, we propose the following simplifying
assumptions of order conditions for the csPRK method
\begin{equation}\label{PRK-simpl-assump}
\begin{split}
&\mathcal {B}(\xi):\;\int_0^1B_\tau
C_\tau^{\kappa-1}\widehat{C_\tau}^{\iota}\,\dif
\tau=\frac{1}{\kappa+\iota},\;\;1\leq
\kappa+\iota\leq\xi;\\
&\mathcal{C}(\eta):\;\int_0^1A_{\tau,\,\sigma}C_\sigma^{\kappa-1}
\widehat{C_\sigma}^{\iota}\,\dif
\sigma=\frac{C_\tau^{\kappa+\iota}}{\kappa+
\iota},\;\;1\leq \kappa+\iota\leq\eta,\;\;\tau\in [0, 1];\\
&\widehat{\mathcal
{C}}(\widehat{\eta}):\;\int_0^1\widehat{A}_{\tau,\,\sigma}C_\sigma^{\kappa-1}
\widehat{C_\sigma}^{\iota}\,\dif
\sigma=\frac{\widehat{C_\tau}^{\kappa+\iota}}{\kappa+\iota}
,\;\; 1\leq \kappa+\iota\leq\widehat{\eta},\;\;\tau\in [0, 1];\\
&\mathcal {D}(\zeta):\;\int_0^1B_\tau
C_\tau^{\kappa-1}\widehat{C_\tau}^{\iota}A_{\tau,\,\sigma}\,\dif
\tau=\frac{B_\sigma(1-
\widehat{C_\sigma}^{\kappa+\iota})}{\kappa+\iota},\;\;
1\leq \kappa+\iota\leq\zeta,\;\;\sigma\in [0, 1];\\
&\widehat{\mathcal {D}}(\widehat{\zeta}):\;\int_0^1\widehat{B}_\tau
C_\tau^{\kappa-1}\widehat{C_\tau}^{\iota}\widehat{A}_{\tau,\,\sigma}\,\dif
\tau= \frac{\widehat{B}_\sigma(1-
\widehat{C_\sigma}^{\kappa+\iota})}{\kappa+\iota},\;\;1\leq
\kappa+\iota\leq\widehat{\zeta},\;\;\sigma\in [0,\,1].
\end{split}
\end{equation}
\begin{thm}\label{cprk:order}
If a csPRK method \eqref{cprk} with coefficients
$(A_{\tau,\,\sigma},\,B_\tau,\,C_\tau;\;\widehat{A}_{\tau,\,\sigma},
\,\widehat{B}_\tau,\,\widehat{C}_\tau)$ satisfies
$\widehat{B}_\tau\equiv B_\tau$ for $\tau\in[0,1]$, and moreover,
all $\mathcal {B}(\rho)$, $\mathcal {C}(\alpha)$, $\mathcal
{\widehat{C}}(\alpha)$, $\mathcal {D}(\beta)$, $\mathcal
{\widehat{D}}(\beta)$ hold,  then it is at least of order
$$\min (\rho,2\alpha+2,\,\alpha+\beta+1).$$
\end{thm}
\begin{proof}
The proof is the same as that for classical (non-partitioned)
Runge-Kutta methods (see \cite{hairernw93sod}), in which the
bi-colored trees have to be considered instead of trees.
\end{proof}

If we let $B_\tau=\widehat{B}_\tau\equiv1$ and
$C_\tau=\widehat{C}_\tau=\tau$, then by using Theorem \ref{mainthm}
it is easy to design $A_{\tau,\,\sigma}$ and
$\widehat{A}_{\tau,\,\sigma}$ such that all the conditions described
in Theorem \ref{cprk:order} are fulfilled. This allows us to
construct lots of csPRK methods of arbitrarily high order.

\subsection{Classical (P)RK method retrieved}

Though we have discussed the construction of csRK and csPRK methods
above, it should be noticed that in general the practical
implementation of cs(P)RK methods needs the use of numerical
quadrature formula. Therefore, next we present classical (P)RK
methods retrieved by using quadrature formulae.

\subsubsection{RK method retrieved}
Applying a quadrature formula denoted by $(b_i,c_i)_{i=1}^r$ to
\eqref{crk}, we derive an $r$-stage classical RK method
\begin{equation}\label{crk:quad}
\begin{split}
&\vec{\overline{Z}}_i=\vec{z}_0+h\sum_{j=1}^r(b_jA_{c_i,\,c_j})\vec{f}(t_0+c_{j}h,
\vec{\overline{Z}}_j),\quad i=1,\cdots,r,\\
&\vec{z}_{1}=\vec{z}_0+h\sum_{i=1}^rb_{i}\vec{f}(t_0+c_{i}h,\vec{\overline{Z}}_i),
\end{split}
\end{equation}
where $A_{c_i,\,c_j}=(A_{\tau,\,\sigma})|_{\tau=c_i,\,\sigma=c_j}$
and $\vec{\overline{Z}}_i
\approx\vec{Z}_{c_i}=(\vec{Z}_{\tau})|_{\tau=c_i}$. The scheme
\eqref{crk:quad} can also be formulated by the following Butcher
tableau

\[\ba{c|c}c & (b_{j}A_{c_i,\,c_j})_{r\times r}\\[4pt]\hline\\[-10pt] & b^T \ea\]
where $c=(c_1,\cdots,c_r)^T,\,b=(b_1,\cdots,b_r)^T$.
\begin{thm}\label{qua:csRK}
Assume $A_{\tau,\,\sigma}$ is a bivariate polynomial of degree
$d^{\tau}$ in $\tau$ and degree $d^{\sigma}$ in $\sigma$, and the
quadrature formula $(b_i,c_i)_{i=1}^r$ is of order\footnote{The
numerical quadrature formula denoted by $(b_i,c_i)_{i=1}^r$ is
called order $p$ if and only if $\int_0^1 f(x)\, \dif
x=\sumli_{i=1}^r b_i f(c_i)$ exactly holds for any polynomial $f(x)$
of degree up to $p-1$.} $p$. If a csRK method \eqref{crk} with
coefficients $(A_{\tau,\,\sigma},\,B_\tau,\,C_\tau)$ satisfies
$B_\tau\equiv1 ,\,C_\tau=\tau$ (then $\mathfrak {B}(\infty)$ holds)
and both $\mathfrak{C}(\eta)$, $\mathfrak{D}(\zeta)$ hold, then the
classical RK method with coefficients $(b_{j}A_{c_i,c_j},b_i,\,c_i)$
is at least of order
$$\min(p, 2\alpha+2, \alpha+\beta+1),$$
where $\alpha=\min(\eta, p-d^{\sigma})$ and $\beta=\min(\zeta,
p-d^{\tau})$.
\end{thm}
\begin{proof}
Since $\int_0^1 f(x)\, \dif x=\sumli_{i=1}^r b_i f(c_i)$ holds
exactly for any polynomial $f(x)$ of degree up to $p-1$, using the
quadrature formula $(b_i,c_i)_{i=1}^r$ to compute the integrals of
$\mathfrak{B}(\xi)$, $\mathfrak{C}(\eta)$, $\mathfrak{D}(\zeta)$ it
gives
\begin{equation*}
\begin{split}
&\sum_{i=1}^rb_ic_i^{\kappa-1}=\frac{1}{\kappa},\;\kappa=1,\cdots,p,\\
&\sum_{j=1}^r(b_{j}A_{c_i,\,c_j})c_j^{\kappa-1}=\frac{c_i^{\kappa}}{\kappa},
\;i=1,\cdots,r,\;\kappa=1,\cdots,\alpha,\\
&\sum_{i=1}^rb_ic_i^{\kappa-1}(b_{j}A_{c_i,\,c_j})=
\frac{b_j}{\kappa}(1-c_j^{\kappa}),\;j=1,\cdots,r,\;\kappa=1,\cdots,\beta.
\end{split}
\end{equation*}
where $\alpha=\min(\eta, p-d^{\sigma})$ and $\beta=\min(\zeta,
p-d^{\tau})$. Note that in the last formula, we have multiplied a
factor $b_j$ afterwards from both sides of the original identity.
These formulae imply that the RK method with coefficients
$(b_{j}A_{c_i,c_j},b_i,\,c_i)$ satisfies $B(p)$, ${C}(\alpha)$ and
${D}(\beta)$ (i.e., the classical simplifying assumptions for RK
method, see \cite{butcher64ipa,hairerw96sod}), which provides the
order of the method from a classical result by Butcher
\cite{butcher64ipa}.
\end{proof}

\subsubsection{PRK method retrieved}

Assume that $B_\tau=\widehat{B}_\tau\equiv1$ and
$C_\tau=\widehat{C}_\tau=\tau$, applying a quadrature formula
denoted by $(b_i,c_i)_{i=1}^r$ to \eqref{cprk}, we then derive a
$r$-stage classical PRK method
\begin{equation}\label{PRK}
\begin{split}
&\vec{\overline{Y}}_i=\vec{y}_0+h\sum_{j=1}^r(b_{j}A_{c_i,\,c_j})
\vec{f}(t_0+c_{j}h,\vec{\overline{Y}}_j,\vec{\overline{Z}}_j),
\quad i=1,\cdots,r,\\
&\vec{\overline{Z}}_i=\vec{z}_0+h\sum_{j=1}^r(b_{j}\widehat{A}_{c_i,\,c_j})
\vec{g}(t_0+c_{j}h,\vec{\overline{Y}}_j,\vec{\overline{Z}}_j),\quad i=1,\cdots,r,\\
&\vec{y}_{1}=\vec{y}_0+h\sum_{i=1}^rb_{i}\vec{f}(t_0+c_{i}h,
\vec{\overline{Y}}_i,\vec{\overline{Z}}_i),\\
&\vec{z}_{1}=\vec{z}_0+h\sum_{i=1}^rb_{i}\vec{g}(t_0+c_{i}h,
\vec{\overline{Y}}_i,\vec{\overline{Z}}_i),
\end{split}
\end{equation}
which can be formulated by a pair of Butcher tableaux below

\[\ba{c|c} c & (b_{j}A_{c_i,\,c_j})_{r\times r}\\[4pt]\hline\\[-10pt]& b^T \ea
\qquad\qquad\qquad\qquad \ba{c|c} c &
(b_{j}\widehat{A}_{c_i,\,c_j})_{r\times
r}\\[4pt]\hline\\[-10pt] & b^T \ea\]
where $c=(c_1,\cdots,c_r)^T,\,b=(b_1,\cdots,b_r)^T$.

\begin{thm}\label{qua:csPRK}
Assume $A_{\tau,\,\sigma}$ is a bivariate polynomial of degree
$d^{\tau}$ in $\tau$ and degree $d^{\sigma}$ in $\sigma$,
$\widehat{A}_{\tau,\,\sigma}$ is a bivariate polynomial of degree
$\widehat{d}^{\tau}$ in $\tau$ and degree $\widehat{d}^{\sigma}$ in
$\sigma$, and the quadrature formula $(b_i,c_i)_{i=1}^r$ is of order
$p$. If
$(A_{\tau,\,\sigma},\,B_\tau,\,C_\tau;\;\widehat{A}_{\tau,\,\sigma},
\,\widehat{B}_\tau,\,\widehat{C}_\tau)$ satisfies
$B_\tau=\widehat{B}_\tau\equiv1 ,\,C_\tau=\widehat{C}_\tau=\tau$
(then $\mathcal {B}(\infty)$ holds) and all $\mathcal {C}(\alpha)$,
$\mathcal {\widehat{C}}(\alpha)$, $\mathcal {D}(\beta)$, $\mathcal
{\widehat{D}}(\beta)$ hold, then the classical PRK method with
coefficients
$(b_{j}A_{c_i,c_j},b_i,\,c_i;\;b_{j}\widehat{A}_{c_i,c_j},b_i,\,c_i)$
is of order
$$\min(p, 2\eta+2, \eta+\zeta+1),$$
where $\eta=\min(\alpha, p-d^{\sigma},p-\widehat{d}^{\sigma})$ and
$\zeta=\min(\beta, p-d^{\tau},p-\widehat{d}^{\tau})$.
\end{thm}
\begin{proof}
The proof of the current theorem is similar to that of Theorem
\ref{qua:csRK}, and one needs using a known result presented in
\cite{cohenh11lei} (see Theorem 4.1, Page 97).
\end{proof}

\section{Construction of symplectic (P)RK methods}

We consider Hamiltonian differential equations written in the
following form
\begin{equation}\label{Part-Hs}
\begin{cases}
\vec{\dot{p}}=-\nabla_{\vec{q}}
\mathbf{H}(\vec{p},\vec{q}),\\[6pt]
\vec{\dot{q}}=\nabla_{\vec{p}} \mathbf{H}(\vec{p},\vec{q}),
\end{cases}
\end{equation}
where the Hamiltonian namely the total energy
$\mathbf{H}(\vec{p},\vec{q})$ is a smooth scalar function of
position coordinates $\vec{p}\in\mathbb{R}^d$ and momenta
$\vec{q}\in\mathbb{R}^d$, and we specify the initial condition as
$(\vec{p}(t_0),\,\vec{q}(t_0))=(\vec{p}_0,\,\vec{q}_0)\in
{\mathbb{R}}^d \times{\mathbb{R}}^d.$ It is well-known that the
energy of the system \eqref{Part-Hs} is a conserved quantity along
the analytic solution of the system, and the system has a geometric
structure called symplecticity
\cite{Arnold89mmo,Feng84ods,Feng95kfc,Fengqq10sga}, i.e., the flow
$\varphi_t$ of the system satisfies
$$(\varphi'_t)^TJ\varphi'_t=J,$$
where $J$ is a canonical structure matrix \cite{hairerlw06gni} and
$\varphi'_t$ denotes the derivative of $\varphi_t$ with respect to
initial values. An important numerical integrator for this special
system is the so-called symplectic method
\cite{Feng84ods,Feng95kfc,hairerlw06gni}, which has an excellent
long-time behavior in simulating the exact flow of the system. A
one-step method $\Phi_h:
(\vec{p}_0,\,\vec{q}_0)\mapsto(\vec{p}_1,\,\vec{q}_1)$ for solving
Hamiltonian systems is called symplectic if
$$(\Phi'_h)^TJ\Phi'_h=J.$$
More surveys and discussions about this special class of algorithms
can be found in \cite{Fengqq10sga,hairerlw06gni,sanzc94nhp} and
references therein.

It is known that for an $r$-stage RK method with coefficients
$(a_{ij},b_i,c_i)$, when applied to the Hamiltonian system
\eqref{Part-Hs}, it can preserve the symplecticity of the system
only if the RK coefficients satisfy
\cite{sanz88rkm,lasagni88crk,suris89ctg}
\begin{equation}\label{cond:SRK}
    b_ia_{ij}+b_ja_{ji}=b_ib_j,\;i,j=1,\cdots,r,
\end{equation}
and for an $r$-stage PRK method with coefficients
$(a_{ij},b_i,c_i;\widehat{a}_{ij},\widehat{b}_i,\widehat{c}_i)$, the
conditions for symplecticity are \cite{sun95coh,hairerlw06gni}
\begin{equation}\label{cond:SPRK}
  \begin{split}
    &b_i\widehat{a}_{ij}+\widehat{b}_ja_{ji}=b_i\widehat{b}_j,\;i,j=1,\cdots,r,\\
    &b_i=\widehat{b}_i,\;i=1,\cdots,r.
  \end{split}
  \end{equation}
Analogously to the classical results above, for a csRK method with
$B_\tau\equiv1,\,C_\tau=\tau$, it is symplectic if
\begin{equation}\label{sym:csRK1}
A_{\tau,\,\sigma}+A_{\sigma,\,\tau}=1,\quad\tau,\,\sigma\in[0,\,1],
\end{equation}
while for a csPRK method with
$B_\tau=\widehat{B}_\tau\equiv1,\,C_\tau=\widehat{C}_\tau=\tau$, it
is symplectic if
\begin{equation}\label{sym:csPRK}
\widehat{A}_{\tau,\,\sigma}+A_{\sigma,\,\tau}=1,\quad\tau,\,\sigma\in[0,\,1].
\end{equation}

Next, let us proceed by considering the construction of symplectic
RK and symplectic PRK type schemes, separately for both
non-classical and classical cases.

\subsection{Symplectic RK type method}

\begin{thm}\label{symthm}
The csRK method \eqref{crk} with $B_\tau=1,\,C_\tau=\tau$ is
symplectic if $A_{\tau,\,\sigma}$ has the form
\begin{equation}\label{sym:csRK2}
A_{\tau,\,\sigma}=\frac{1}{2}+\sum_{0<i+j\in \mathbb{Z}}\alpha_{ij}
P_i(\tau)P_j(\sigma),\quad\alpha_{ij}\in \mathbb{R},
\end{equation}
where $\alpha_{ij}$ is skew-symmetric, i.e.,
$\alpha_{ij}=-\alpha_{ji},\,i+j>0$.
\end{thm}
\begin{proof}
Assume $A_{\tau,\,\sigma}$ can be expanded as a series in terms of
basis $\left\{P_i(\tau)P_j(\sigma)\right\}_{i,j=0}^\infty$, written
in the form
\begin{equation*}
A_{\tau,\,\sigma}=\sum_{0\leq i,j\in\mathbb{Z}}\alpha_{ij}
P_i(\tau)P_j(\sigma),\quad\alpha_{ij}\in \mathbb{R},
\end{equation*}
and thus we have
\begin{equation*}
A_{\sigma,\,\tau}=\sum_{0\leq i,j\in\mathbb{Z}}\alpha_{ij}
P_i(\sigma)P_j(\tau) =\sum_{0\leq i,j\in\mathbb{Z}}\alpha_{ji}
P_j(\sigma)P_i(\tau).
\end{equation*}
Substituting the above two expressions into \eqref{sym:csRK1} and
collecting the like basis, it gives
$$\alpha_{00}=\frac{1}{2};\;\,\alpha_{ij}=-\alpha_{ji},\,i+j>0,$$
which completes the proof.
\end{proof}
It is evident that combing Theorem \ref{symthm} with Theorem
\ref{mainthm}, we can construct symplectic csRK methods of arbitrary
order. Moreover, when consider using a quadrature formula, we have
the following corollary.
\begin{coro}
If the coefficient $A_{\tau,\,\sigma}$ of a csRK method with
$B_\tau=1,\,C_\tau=\tau$ are given by \eqref{sym:csRK2}, then the RK
scheme $(b_{j}A_{c_i,c_j},b_i,\,c_i)$ obtained by using a quadrature
formula $(b_i,c_i)_{i=1}^r$ is always symplectic and the order of
the method can be determined by Theorem \ref{qua:csRK}.
\end{coro}
\begin{proof}
When applying a quadrature formula $(b_i,c_i)_{i=1}^r$ to calculate
the integrals of the associated csRK method, it can be seen that
$$A_{c_i,\,c_j}+A_{c_j,\,c_i}=1,\;\;i,j=1,\cdots,r,$$ thus
$$b_i(b_jA_{c_i,\,c_j})+b_j(b_iA_{c_j,\,c_i})=b_ib_j,\;\;i,j=1,\cdots,r,$$
which implies that the conditions for symplecticity, namely
\eqref{cond:SRK}, are fulfilled.
\end{proof}
\begin{exa}
Let $B_\tau=1,\,C_\tau=\tau$ and take
\begin{equation}\label{exa1}
A_{\tau,\,\sigma}=\frac{1}{2}+\sum_{\iota=0}^{s-1}\xi_{\iota+1}
\Big(P_{\iota+1}(\tau)P_\iota(\sigma)-P_{\iota+1}(\sigma)P_\iota(\tau)\Big),
\end{equation}
by using some suitable quadrature formulae, we can retrieve a lot of
classical symplectic RK methods including $(s+1)$-stage Lobatto
IIIE, Radau IB, Radau IIB and Gauss-Legendre RK schemes, and the
orders are $2s,2s+1,2s+1,2s+2$ respectively.

If we introduce a parameter $\lambda$ by adding a term
$\lambda\xi_{s+1}\big(P_{s+1}(\tau)P_s(\sigma)-P_{s+1}(\sigma)P_s(\tau)\big)$
to \eqref{exa1}, and consider $s=1$, then by using 3-point Gaussian
quadrature formula it produces a new class of symplectic RK methods,
which is shown in Table \ref{exa1:SRK}.
\end{exa}

\subsection{Symplectic PRK type method}

For convenience, we first give the following lemma.
\begin{lem}\label{lem1}
Assume $B_\tau\equiv1$ and $C_\tau=\tau$. If
$(A_{\tau,\,\sigma},\,B_\tau,\,C_\tau)$ satisfies
$\mathfrak{C}(\eta)$ and $\mathfrak{D}(\zeta)$, define
$\widehat{A}_{\tau,\,\sigma}:=1-A_{\sigma,\,\tau}$, then
$(\widehat{A}_{\tau,\,\sigma},\,B_\tau,\,C_\tau)$ satisfies
$\mathfrak{C}(\zeta)$ and $\mathfrak{D}(\eta)$.
\end{lem}
\begin{proof}
The proof can be directly deduced from Theorem \ref{mainthm}.
\end{proof}
\begin{thm}\label{symthm:csPRK}
If a csPRK method \eqref{cprk} with coefficients
$(A_{\tau,\,\sigma},\,B_\tau,\,C_\tau;\;\widehat{A}_{\tau,\,\sigma},
\,\widehat{B}_\tau,\,\widehat{C}_\tau)$ satisfying
$$B_\tau=\widehat{B}_\tau\equiv1,$$
$$C_\tau=\widehat{C}_\tau=\tau,$$
$$\widehat{A}_{\tau,\,\sigma}=1-A_{\sigma,\,\tau},$$
and $A_{\tau,\,\sigma}$ is given such that both $\mathfrak{C}(\eta)$
and $\mathfrak{D}(\zeta)$ hold, then this csPRK method is symplectic
and of order at least
$$2\min(\eta,\,\zeta)+1.$$
\end{thm}
\begin{proof}
Since $B_\tau=\widehat{B}_\tau\equiv1,C_\tau=\widehat{C}_\tau=\tau$,
this implies that $\mathcal{B}(\infty)$ holds. In such a case, note
that $\mathfrak{C}(\eta)$ and $\mathfrak{D}(\zeta)$ are equivalent
to $\mathcal{C}(\eta)$ and $\mathcal{D}(\zeta)$, respectively. Next,
by combining Theorem \ref{cprk:order} and Lemma \ref{lem1} it gives
the final result.
\end{proof}
One can then easily construct symplectic PRK methods by means of
Theorem \ref{symthm:csPRK} and Theorem \ref{qua:csPRK} based on
numerical quadrature formulae. we provide the following corollary.
\begin{coro}\label{coro:PRK}
Assume $A_{\tau,\,\sigma}$ is a bivariate polynomial of degree
$d^{\tau}$ in $\tau$ and degree $d^{\sigma}$ in $\sigma$,
$\widehat{A}_{\tau,\,\sigma}=1-A_{\sigma,\,\tau}$, and the
quadrature formula $(b_i,c_i)_{i=1}^r$ is of order $p$. If
$(A_{\tau,\,\sigma},\,B_\tau,\,C_\tau)$ satisfies $B_\tau\equiv1
,\,C_\tau=\tau$ (then $\mathfrak{B}(\infty)$ holds) and both
$\mathfrak{C}(\eta)$, $\mathfrak{D}(\zeta)$ hold, then the classical
PRK method with coefficients
$(b_{j}A_{c_i,c_j},b_i,\,c_i;\;b_{j}\widehat{A}_{c_i,c_j},b_i,\,c_i)$
is symplectic and of order at least
$$\min(p, 2\alpha+1),$$
where $\alpha=\min(\eta,\,\zeta,\, p-d^{\tau},\,p-d^{\sigma})$.
\end{coro}
\begin{proof}
The symplecticity-preserving property of the associated PRK method
can be easily seen from the following fact
$$b_i(b_j\widehat{A}_{c_i,\,c_j})+b_j(b_iA_{c_j,\,c_i})=b_ib_j,\;\;i,j=1,\cdots,r.$$
Note that $\widehat{A}_{\tau,\,\sigma}=1-A_{\sigma,\,\tau}$ is a
bivariate polynomial of degree $d^{\sigma}$ in $\tau$ and degree
$d^{\tau}$ in $\sigma$, thus inserting
$\widehat{d}^{\tau}=d^{\sigma}$ and $\widehat{d}^{\sigma}=d^{\tau}$
into the formula given in Theorem \ref{qua:csPRK} gives the final
result.
\end{proof}
\begin{table}
\[\ba{c|ccc}\frac{1}{2}-\frac{\sqrt{15}}{10} &  \frac{5}{36} &
\frac{2}{9}-\frac{(2+\lambda)\sqrt{15}}{45} &
\frac{5}{36}-\frac{(5-2\lambda)\sqrt{15}}{90}\\[2pt]
\frac{1}{2} & \frac{5}{36}+\frac{(2+\lambda)\sqrt{15}}{72}&
\frac{2}{9} &
\frac{5}{36}-\frac{(2+\lambda)\sqrt{15}}{72} \\[2pt]
\frac{1}{2}+\frac{\sqrt{15}}{10} &
\frac{5}{36}+\frac{(5-2\lambda)\sqrt{15}}{90} &
\frac{2}{9}+\frac{(2+\lambda)\sqrt{15}}{45} & \frac{5}{36}\\[2pt]
\hline &\frac{5}{18} & \frac{4}{9} & \frac{5}{18}\ea
\]
\caption{A class of symplectic RK methods of order 4 retrieved by
Gaussian quadrature.}\label{exa1:SRK}
\end{table}
\begin{table}
\[\ba{c|ccc} \frac{5-\sqrt{15}}{10} &
\frac{5}{36}-\frac{\sqrt{15}}{90} &
\frac{2}{9}-\frac{2\sqrt{15}}{45}&
\frac{5}{36}-\frac{2\sqrt{15}}{45}\\[2pt]
\frac{1}{2} & \frac{5}{36}+\frac{\sqrt{15}}{24}
& \frac{2}{9}& \frac{5}{36}-\frac{\sqrt{15}}{24}\\[2pt]
\frac{5+\sqrt{15}}{10} & \frac{5}{36}+\frac{2\sqrt{15}}{45} &
\frac{2}{9}+\frac{2\sqrt{15}}{45}
& \frac{5}{36}+\frac{\sqrt{15}}{90}\\[2pt]
\hline & \frac{5}{18} & \frac{4}{9} & \frac{5}{18} \ea\qquad
\ba{c|ccc} \frac{5-\sqrt{15}}{10} &
\frac{5}{36}+\frac{\sqrt{15}}{90} &
\frac{2}{9}-\frac{\sqrt{15}}{15}&
\frac{5}{36}-\frac{2\sqrt{15}}{45}\\[2pt]
\frac{1}{2} & \frac{5}{36}+\frac{\sqrt{15}}{36} & \frac{2}{9}&
\frac{5}{36}-\frac{\sqrt{15}}{36}\\[2pt]
\frac{5+\sqrt{15}}{10} & \frac{5}{36}+\frac{2\sqrt{15}}{45} &
\frac{2}{9}+\frac{\sqrt{15}}{15}&
\frac{5}{36}-\frac{\sqrt{15}}{90}\\[2pt]
\hline & \frac{5}{18} & \frac{4}{9} & \frac{5}{18} \ea\] \caption{A
new symplectic PRK method of order 4 retrieved by Gaussian
quadrature.}\label{exa2:SPRK1}
\end{table}
\begin{table}
\[\ba{c|ccc}0 & 0 & 0 & 0 \\[2pt]
\frac{6-\sqrt{6}}{10} & \frac{19}{150}-\frac{\sqrt{6}}{225} &
\frac{1}{4}+\frac{\sqrt{6}}{72}& \frac{67}{300}-\frac{197\sqrt{6}}{1800}\\[2pt]
\frac{6+\sqrt{6}}{10} & \frac{19}{150}+\frac{\sqrt{6}}{225}&
\frac{67}{300}+\frac{197\sqrt{6}}{1800}&
\frac{1}{4}-\frac{\sqrt{6}}{72}\\[2pt]
\hline & \frac{1}{9}& \frac{16+\sqrt{6}}{36}& \frac{16-\sqrt{6}}{36}
\ea\qquad \ba{c|ccc}0 & \frac{1}{9} &
-\frac{1}{18}+\frac{\sqrt{6}}{72}& -\frac{1}{18}-\frac{\sqrt{6}}{72} \\[2pt]
\frac{6-\sqrt{6}}{10} & \frac{1}{9}&\frac{7}{36}+\frac{\sqrt{6}}{72}
&\frac{53}{180}-\frac{41\sqrt{6}}{360}\\[2pt]
\frac{6+\sqrt{6}}{10} & \frac{1}{9}&
\frac{53}{180}+\frac{41\sqrt{6}}{360}&
\frac{7}{36}-\frac{\sqrt{6}}{72}\\[2pt]
\hline & \frac{1}{9}& \frac{16+\sqrt{6}}{36}& \frac{16-\sqrt{6}}{36}
\ea\] \caption{A new symplectic PRK method of order 4 retrieved by
Radau-Left quadrature.}\label{exa2:SPRK2}
\end{table}
\begin{table}
\[\ba{c|ccc} \frac{4-\sqrt{6}}{10} & \frac{14-\sqrt{6}}{72}&
\frac{398-147\sqrt{6}}{1800}& \frac{-7-2\sqrt{6}}{450}\\[2pt]
\frac{4+\sqrt{6}}{10} & \frac{398+147\sqrt{6}}{1800} &
\frac{14+\sqrt{6}}{72}& \frac{-7+2\sqrt{6}}{450}\\[2pt]
1 & \frac{16-\sqrt{6}}{36}& \frac{16+\sqrt{6}}{36}& \frac{1}{9}\\[2pt]
\hline & \frac{16-\sqrt{6}}{36}& \frac{16+\sqrt{6}}{36}& \frac{1}{9}
\ea\qquad \ba{c|ccc} \frac{4-\sqrt{6}}{10} &
\frac{1}{4}-\frac{\sqrt{6}}{72}&
\frac{3}{20}-\frac{31\sqrt{6}}{360}& 0\\[2pt]
\frac{4+\sqrt{6}}{10} & \frac{3}{20}+\frac{31\sqrt{6}}{360} &
\frac{1}{4}+\frac{\sqrt{6}}{72}& 0\\[2pt]
1 & \frac{1}{2}-\frac{\sqrt{6}}{72}& \frac{1}{2}+\frac{\sqrt{6}}{72}& 0\\[2pt]
\hline & \frac{16-\sqrt{6}}{36}& \frac{16+\sqrt{6}}{36}& \frac{1}{9}
\ea\] \caption{A new symplectic PRK method of order 4 retrieved by
Radau-Right quadrature.}\label{exa2:SPRK3}
\end{table}
\begin{table}
\[\ba{c|cccc}
0 & 0 & 0 & 0& 0\\[2pt]
\frac{5-\sqrt{5}}{10} &\frac{11-\sqrt{5}}{120} &
\frac{25+\sqrt{5}}{120} & \frac{25-11\sqrt{5}}{120}&
\frac{-1-\sqrt{5}}{120}\\[2pt]
\frac{5+\sqrt{5}}{10} &\frac{11+\sqrt{5}}{120}
&\frac{25+11\sqrt{5}}{120}&\frac{25-\sqrt{5}}{120}
&\frac{-1+\sqrt{5}}{120}\\[2pt]
1 &\frac{1}{12} & \frac{5}{12} & \frac{5}{12}&\frac{1}{12}\\[2pt]
\hline &\frac{1}{12} & \frac{5}{12} &
\frac{5}{12}&\frac{1}{12}\ea\qquad \ba{c|cccc} 0 &\frac{1}{12} &
\frac{-1+\sqrt{5}}{24} & \frac{-1-\sqrt{5}}{24}& 0\\[2pt]
\frac{5-\sqrt{5}}{10} & \frac{1}{12}
&\frac{25-\sqrt{5}}{120}&\frac{25-11\sqrt{5}}{120}&0\\[2pt]
\frac{5+\sqrt{5}}{10} &\frac{1}{12}
&\frac{25+11\sqrt{5}}{120}&\frac{25+\sqrt{5}}{120}&0\\[2pt]
1 &\frac{1}{12}&\frac{11+\sqrt{5}}{24}&\frac{11-\sqrt{5}}{24}&0\\[2pt]
\hline &\frac{1}{12} & \frac{5}{12} & \frac{5}{12}&\frac{1}{12}\ea\]
\caption{A new symplectic PRK method of order 4 retrieved by Lobatto
quadrature.}\label{exa2:SPRK4}
\end{table}
\begin{table}
\[\ba{c|ccc} \frac{5-\sqrt{15}}{10} & \frac{2}{9} &
\frac{2}{9}-\frac{2\sqrt{15}}{45}&
\frac{1}{18}-\frac{\sqrt{15}}{18}\\[2pt]
\frac{1}{2} & \frac{5}{36}+\frac{\sqrt{15}}{36}& \frac{2}{9}&
\frac{5}{36}-\frac{\sqrt{15}}{36}\\[2pt]
\frac{5+\sqrt{15}}{10}& \frac{1}{18}+\frac{\sqrt{15}}{18}
& \frac{2}{9}+\frac{2\sqrt{15}}{45}&\frac{2}{9}\\[2pt]
\hline & \frac{5}{18} & \frac{4}{9} & \frac{5}{18} \ea\qquad
\ba{c|ccc} \frac{5-\sqrt{15}}{10} & \frac{1}{18} &
\frac{2}{9}-\frac{2\sqrt{15}}{45}&
\frac{2}{9}-\frac{\sqrt{15}}{18}\\[2pt]
\frac{1}{2} & \frac{5}{36}+\frac{\sqrt{15}}{36} & \frac{2}{9}&
\frac{5}{36}-\frac{\sqrt{15}}{36}\\[2pt]
\frac{5+\sqrt{15}}{10} & \frac{2}{9}+\frac{\sqrt{15}}{18}
& \frac{2}{9}+\frac{2\sqrt{15}}{45}& \frac{1}{18}\\[2pt]
\hline & \frac{5}{18} & \frac{4}{9} & \frac{5}{18}\ea\] \caption{A
new symplectic PRK method of order 3 retrieved by Gaussian
quadrature.}\label{exa2:SPRK5}
\end{table}
\begin{exa}
Let $B_\tau=\widehat{B}_\tau\equiv1,\,C_\tau=\widehat{C}_\tau=\tau$
and take
\begin{equation}\label{exa2}
\begin{split}
A_{\tau,\,\sigma}& =\frac{1}{2}+\sum_{\iota=0}^{s-1}\xi_{\iota+1}
P_{\iota+1}(\tau)P_\iota(\sigma)-\sum_{\iota=0}^{s-2}\xi_{\iota+1}
P_{\iota+1}(\sigma)P_\iota(\tau),\\
\widehat{A}_{\tau,\,\sigma}&=1-A_{\sigma,\,\tau}=\frac{1}{2}+\sum_{\iota=0}^{s-2}\xi_{\iota+1}
P_{\iota+1}(\tau)P_\iota(\sigma)-\sum_{\iota=0}^{s-1}\xi_{\iota+1}
P_{\iota+1}(\sigma)P_\iota(\tau),
\end{split}
\end{equation}
by using some suitable quadrature formulae with $s$ quadrature
points, we can retrieve a lot of classical high order symplectic PRK
methods including $s$-stage Lobatto IIIA-IIIB, Radau IA-I\={A},
Radau IIA-II\={A} and Gauss-Legendre schemes, and the orders are
$2s-2,2s-1,2s-1,2s$ respectively.

If we use more quadrature points, then we can get lots of new
symplectic PRK schemes which are not found in the literature. For
instance, consider $s=2$,  it gives the symplectic PRK schemes shown
in Table \ref{exa2:SPRK1}-\ref{exa2:SPRK4}.
\end{exa}
\begin{exa}
Let $B_\tau=\widehat{B}_\tau\equiv1,\,C_\tau=\widehat{C}_\tau=\tau$
and take
\begin{equation}\label{exa3}
\begin{split}
A_{\tau,\,\sigma}& =\frac{1}{2}+\sum_{\iota=0}^{s-1}\xi_{\iota+1}
\Big(P_{\iota+1}(\tau)P_\iota(\sigma)-P_{\iota+1}(\sigma)P_\iota(\tau)
\Big)+\frac{1}{2(2s+1)}P_s(\tau)P_s(\sigma),\\
\widehat{A}_{\tau,\,\sigma}&=1-A_{\sigma,\,\tau}=\frac{1}{2}+\sum_{\iota=0}^{s-1}\xi_{\iota+1}
\Big(P_{\iota+1}(\tau)P_\iota(\sigma)-P_{\iota+1}(\sigma)P_\iota(\tau)
\Big)-\frac{1}{2(2s+1)}P_s(\tau)P_s(\sigma),
\end{split}
\end{equation}
by using suitable quadrature formulae with $s+1$ quadrature points,
we can retrieve a lot of classical high order symplectic PRK methods
including $(s+1)$-stage Lobatto IIIC-III\={C}, Radau IA-I\={A},
Radau IIA-II\={A} and Gauss IA-I\={A} schemes\footnote{These high
order symplectic PRK schemes were firstly constructed by G. Sun in
1995 \cite{sun95coh}.}, and the orders are $2s,2s+1,2s+1,2s+1$
respectively.

If we use more quadrature formulae with different number of
quadrature points, then we can get more new symplectic PRK schemes.
Table \ref{exa2:SPRK5} provides a symplectic PRK scheme based on
Gaussian quadrature for the case $s=1$.
\end{exa}

\section{Conclusions}

This paper investigates the construction of symplectic (P)RK type
methods based on continuous-stage (P)RK methods. In the process of
construction, we first discuss the general continuous-stage (P)RK
methods with the help of the simplifying assumptions and the
expansion technique of orthogonal polynomials (i.e., Legendre
polynomials), and then establish (P)RK methods in classical sense by
means of numerical quadrature formulae. The newly obtained results
are then directly applied to construct a special class of (P)RK
methods which have symplecticity-preserving property for solving
Hamiltonian systems. The line of construction of numerical
algorithms combines the continuous and discrete mathematical theory,
which forms a new and simple way for numerical integration of ODEs
especially for Hamiltonian systems with a symplectic structure.

\section*{Acknowledgments}

The first author was supported by the National Natural Science
Foundation of China (11401055) and the Foundation of Education
Department of Hunan Province (15C0028). The third author was
supported by the NNSFC (11401054).





\begin{thebibliography}{10}

\bibitem{Arnold89mmo} {V.I. Arnold},{ \em Mathematical methods of classical
mechanics}, Vol. 60, Springer, 1989.

\bibitem{brugnanoit10hbv}{L.~Brugnano, F.~Iavernaro, D.~Trigiante}, {\em
Hamiltonian boundary value methods: energy preserving discrete line
integral methods}, J. Numer. Anal., Indust. Appl. Math., 5 (1--2)
(2010), 17--37.

\bibitem{brugnanoit15aoh}{L. Brugnano, F. Iavernaro, D. Trigiante}, {\em
Analysis of Hamiltonian Boundary Value Methods (HBVMs): A class of
energy-preserving Runge-Kutta methods for the numerical solution of
polynomial Hamiltonian systems}, Commun. Nonlinear Sci., 20 (2015),
650--667.

\bibitem{butcher72aat}{J.~C.~Butcher}, {\em An algebraic theory of
integration methods}, Math. Comp., 26 (1972), 79--106.

\bibitem{butcher64ipa}{J.~C.~Butcher}, {\em Implicit Runge-Kutta processes},
Math. Comput. 18 (1964),  50--64.

\bibitem{butcher87tna}{J.~C.~Butcher}, {\em The Numerical Analysis of Ordinary
Differential Equations: Runge-Kutta and General Linear Methods},
John Wiley \& Sons, 1987.

\bibitem{Celledoni09mmoqw}{E.~Celledoni, R.~I.~McLachlan, D.~McLaren, B.~Owren,
G.~R.~W.~Quispel, W.~M.~Wright.}, {\em Energy preserving Runge-Kutta
methods}, M2AN 43 (2009), 645--649.


\bibitem{cohenh11lei}{D.~Cohen, E.~Hairer}, {\em Linear energy-preserving
integrators for Poisson systems}, BIT. Numer. Math., 51(2011),
91--101.

\bibitem{Feng84ods}{K. Feng}, {\em On difference schemes and symplectic geometry},
Proceedings of the 5-th Inter., Symposium of Differential Geometry
and Differential Equations, Beijing, 1984, 42-58.

\bibitem{Feng95kfc}{K.~Feng}, {\em K. Feng's Collection of Works}, Vol. 2,
Beijing: National Defence Industry Press, 1995.

\bibitem{Fengqq10sga}{K. Feng, M. Qin}, {\em Symplectic Geometric Algorithms for
Hamiltonian Systems}, Spriger and Zhejiang Science and Technology
Publishing House, Heidelberg, Hangzhou, First edition, 2010.

\bibitem{hairernw93sod}{E.~Hairer, S.~P.~N\o rsett, G.~Wanner}, {\em
Solving Ordiary Differential Equations I: Nonstiff Problems},
Springer Series in Computational Mathematics, 8, Springer-Verlag,
Berlin, 1993.

\bibitem{hairerw96sod}{E.~Hairer, G.~Wanner}, {\em Solving Ordiary
Differential Equations II: Stiff and Differential-Algebraic
Problems}, Second Edition, Springer Series in Computational
Mathematics, 14, Springer-Verlag, Berlin, 1996.

\bibitem{hairerlw06gni}{E.~Hairer, C.~Lubich, G.~Wanner}, {\em Geometric
Numerical Integration: Structure-Preserving Algorithms For Ordinary
Differential Equations}, Second edition, Springer Series in
Computational Mathematics, 31, Springer-Verlag, Berlin, 2006.

\bibitem{hairer10epv}{E.~Hairer}, {\em Energy-preserving variant of
collocation methods}, JNAIAM J. Numer. Anal. Indust. Appl. Math., 5
(2010), 73--84.

\bibitem{Iavernarop07sst}{F.~Iavernaro, B.~Pace}, {\em s-stage trapezoidal
methods for the conservation of Hamiltonian functions of polynomial
type}, AIP Conf. Proc., 936 (2007), 603--606.

\bibitem{lasagni88crk}{F.~Lasagni}, {\em Canonical Runge-Kutta methods},
ZAMP, 39 (1988), 952--953.

\bibitem{miyatake14aee}{Y.~Miyatake}, {\em An energy-preserving exponentially-fitted
continuous stage Runge-Kutta methods for Hamiltonian systems}, BIT
Numer. Math., DOI 10.1007/s10543-014-0474-4, 2014.

\bibitem{quispelm08anc}{G.~R.~W.~Quispel, D.~I.~McLaren}, {\em A new class of
energy-preserving numerical integration methods}, J. Phys. A: Math.
Theor., 41 (2008) 045206.

\bibitem{sanz88rkm}{J.~M.~Sanz-Serna}, {\em Runge-Kutta methods for Hamiltonian
systems}, BIT, 28 (1988), 877--883.

\bibitem{sanzc94nhp}{J.~M.~Sanz-Serna,  M.~P.~Calvo}, {\em Numerical Hamiltonian
problems}, Chapman \& Hall, 1994.

\bibitem{sun93coh}{G.~Sun}, {\em Construction of high order symplectic Runge-Kutta
methods}, J. Comput. Math., 11(4) 1993, 250--260.

\bibitem{sun95coh}{G.~Sun}, {\em Construction of high order symplectic PRK
methods}, J. Comput. Math., 13 (1) 1995, 40--50.

\bibitem{suris89ctg}{Y.~B.~Suris}, {\em Canonical transformations generated by methods
of Runge-Kutta type for the numerical integration of the system
$x''=-\frac{\partial U}{\partial x}$}, Zh. Vychisl. Mat. iMat. FiZ.,
29 (1989), 202--211.

\bibitem{Tangs12tfe}{W.~Tang, Y.~Sun}, {\em Time finite element methods: A
unified framework for numerical discretizations of ODEs}, Appl.
Math. Comput., 219 (2012), 2158--2179.

\bibitem{Tangs12ana}{W.~Tang, Y.~Sun}, {\em A new approach to construct Runge-Kutta
type methods and geometric numerical integrators}, AIP. Conf. Proc.,
1479 (2012), 1291-1294.

\bibitem{Tang13tfe}{W.~Tang}, {\em Time finite element methods, continuous-stage
Runge-Kutta methods and structure-preserving algorithms}, PhD
thesis, AMSS, Chinese Academy of Sciences, 2013.

\bibitem{Tangs14cor}{W.~Tang, Y.~Sun}, {\em Construction of Runge-Kutta type
methods for solving ordinary differential equations}, Appl. Math.
Comput., 234 (2014), 179--191.

\end{thebibliography}




\end{document}